\documentclass{amsart}

\usepackage{amssymb,comment,amsmath,amscd,amsthm,enumerate,verbatim,color}
\usepackage[utf8]{inputenc}

\usepackage{graphicx}
\usepackage{pdfsync}
\usepackage{subfigure}
\usepackage{tikz}
\usepackage[linktocpage=true]{hyperref}

\hypersetup{ colorlinks   = true, %Colours links instead of ugly boxes
urlcolor  = blue, %Colour for external hyperlinks
linkcolor    = black, %Colour of internal links
citecolor   = black %Colour of citations
}

\def\t{{\theta}}
\def\l{{\lambda}}

\def\a{{\alpha}}
\def\e{{\varepsilon}}
\def\beq{\begin{equation}}
\def\eeq{\end{equation}}

\newcommand{\Z}{{\mathbb Z}}
\newcommand{\R}{{\mathbb R}}

\newcommand{\C}{{\mathbb C}}

\newcommand{\N}{{\mathbb N}}
\newcommand{\PP}{{\mathbb P}}

\newcommand{\CC}{{\mathcal C}}
\newcommand{\DD}{{\mathcal D}}

\newcommand{\CB}{{\mathcal B}}

\newcommand{\CI}{{\mathcal I}}

\newcommand{\CU}{{\mathcal U}}

\newcommand{\eqdef}{\overset{\mathrm{def}}=}

\newtheorem{theorem}{Theorem}

\newtheorem{lemma}{Lemma}

\newtheorem{corollary}{Corollary}

\newtheorem*{problem}{Problem}

\allowdisplaybreaks

\numberwithin{equation}{section}

\begin{document}

\title[Positivity of LE for doubling map]{Uniform positivity of the Lyapunov exponent for monotone potentials generated by the doubling map}

\author[Z.\ Zhang]{Zhenghe Zhang}

\address{Department of Mathematics, Rice University, Houston, TX~77005, USA}

\email{zzhang@rice.edu}

\thanks{The author was in part supported by AMS-Simons travel grant 2014-2016.}

\begin{abstract} We show that for any doubling map generated $C^1$ monotone potential with derivative uniformly bounded away from zero, the Lyapunov exponent of the associated Schr\"odinger operators is uniformly positive for all energies provided the coupling constant is large.
\end{abstract}

\maketitle

\setcounter{tocdepth}{1}
%\tableofcontents

\section{Introduction}\label{s:introduction}

Consider the family of Schr\"odinger operators $H_{\l,v,x}: \ell^2(\N)\rightarrow \ell^2(\N)$
\beq\label{eq:operators}
(H_{\lambda,v,x}u)_n=u_{n+1}+u_{n-1}+\lambda v(2^nx)u_n
\eeq
with Dirichlet condition $u_{-1}=0$. Here $v:\R/\Z\rightarrow \R$ is a bounded measurable function is called the potential, $\lambda\in\R$ the coupling constant, $x\in\R/\Z$ the phase. For simplicity, we leave $\l, v$ in $H_{\lambda,v,x}$ implicit.

Let $T:\R/\Z\rightarrow\R/\Z$ be the doubling map $T(x)=2x$. Consider the eigenvalue equation $H_{x}u=Eu.$ Then there is an associated cocycle map denoted
$A^{(E-\lambda v)}:\R/\Z\rightarrow \mathrm{SL}(2,\R)$, and is given by
\beq\label{eq:cocycle:map}
A^{(E-\lambda v)}(x)=\begin{pmatrix}E-\lambda v(x)& -1\\1& 0\end{pmatrix}
\eeq
which is called a Schr\"odinger cocycle map. Then
$$
(T, A^{(E-\lambda v)}):(\R/\Z)\times\R^2\rightarrow(\R/\Z)\times\R^2
$$
defines a family of (non-invertible) dynamical systems:
\beq\label{eq:schrodinger:cocycle}
(x,\vec w)\mapsto (Tx, A^{(E-\l v)}(x)\vec w)
\eeq
which is called the Schr\"odinger cocycle. The $n$th iteration of the dynamics is denoted by $(\alpha, A^{(E-\l v)})^n=(n\alpha, A^{(E-\l v)}_n)$, $n\ge 0$. Thus, we have the following cocycle iterations:
$$
A^{(E-\l v)}_n(x)=\begin{cases}A^{(E-\l v)}(T^{n-1}x)\cdots A^{(E-\l v)}(x), &n\ge 1,\\ I_2, &n=0.\end{cases}
$$

The relation between operator and cocycle is the following. A complex valued sequence, $u=(u_n)_{n\in\N}\in\C^{\N}$ (not necessarily in $\ell^2(\N)$), is a solution of the equation $H_{x}u=Eu$ if and only if
$$
A^{(E-\l v)}_n(x)\binom{u_{0}}{u_{-1}}=\binom{u_{n}}{u_{n-1}},\ n\in\N.
$$
This says that the iteration of cocycle map generates the transfer matrices for the operator (\ref{eq:operators}).

One of the most important objects in the study of dynamics of Schr\"odinger cocycles and spectral properties of the Schrodinger operators is the Lyapunov exponent:
\begin{align}\label{eq:le}
L(E;\l)&=\lim_{n\rightarrow\infty}\frac1n\int_{\R/\Z}\log\|A^{(E-\l v)}_n(x)\|dx\\
\nonumber&=\inf_{n\ge1}\frac1n\int_{\R/\Z}\log\|A^{(E-\l v)}_n(x)\|dx\\
\nonumber&\ge 0.
\end{align}
 The limit in \eqref{eq:le} exists and is equal to the infimum since the sequence
 $$
 \left\{\int_{\R/\Z}\log\|A^{(E-\l v)}_n(x)\|dx\right\}_{n\ge 0}
 $$
is subadditive. Moreover, by Kingman's Subadditive Ergodic Theorem, we also have:
$$
L(E;\l)=\lim\limits_{n\rightarrow\infty}\frac{1}{n}\log\|A^{(E-\l v)}_n(x)\| \mbox{ for } a.e.\ x\in\R/\Z.
$$

\subsection{Statement of main result and some review}\label{ss:main:result}
Throughout this paper, $C,\ c$ are some universal positive constants depending only on $v$, where $C$ is large and $c$ small. Assume that $v:\R/\Z\rightarrow\R$ satisfies the following conditions:

Let $v$ be $C^1$ on $(0,1)$. We further assume that $\|v\|_{C^1}<C$ and $\inf_{x\in(0,1)} |v'|>c$. Hence, we may define $v(0):=\lim_{x\to 0+}v(x)$. Notice that $v$ is discontinuous at $x=0$.

Then, we show the following result.
\begin{theorem}\label{t.main}
  Let $v$ be as above. Consider the corresponding family of Schr\"odinger operators (\ref{eq:operators}). Then there exists a $C_0=C_0(v)>0$ such that for all $\l>0$,
\begin{center}
$L(E;\l)>\log\l-C_0$, for all $E\in\R$.
\end{center}

\end{theorem}
Clearly, the Lyapunov exponent is uniform bounded away from zero for all energies (uniform positivity) provided the coupling constant $\l$ is large.

Lyapunov exponent is a central object in both dynamical systems and spectral analysis of the one-dimensional ergodic Sch\"odinger operators, see e.g. \cite{wilkinson} for an excellent and concise introduction.

Uniform positivity in particular is a strong indication of Anderson Localization (i.e. pure point spectrum with exponentially decaying eigenfunctions) for almost every phase, and is considered to be difficult to get for whatever type of base dynamics.

In fact, most of the existing results are for two extremal cases in terms of randomness: one is for potentials given by i.i.d random variables where uniform positivity for all nonzero couplings was essentially established by Furstenberg \cite{furstenberg}; the other one is for quasiperiodic potentials where uniform positivity holds at large couplings for a large class of potentials. For base dynamics with intermediate randomness, it's more difficult since few tools are available. We refer the readers to \cite[Section 1.2]{wangzhang} for a more detailed review of the related results. In the following we focus on the case of the doubling map.

Doubling map clearly behaves more like i.i.d random case since it's strongly mixing. In particular, it's expected that uniform positivity holds for all couplings for general potentials. However, the only known examples that uniform positivity holds so far are potentials given by trigonometric polynomials where Herman's \cite{herman} trick applies. This, when indeed happens, happens for a very specific reason. For instance, unlike the quasiperiodic case, the use of subharmonicity breaks down completely if one moves away from trigonometric polynomials. See Appendix Section~\ref{s:appendix.b} for a proof of uniform positivity of the Lyapunov exponent for trigonometric polynomials and for a more detailed comments why it breaks down away from trigonometric polynomials.

For partial results, most existing tools are only able to deal with couplings that are not too large. For small coupling, Chulaevsky-Spencer \cite{chulaevskyspencer} first showed that $L(E)$ is uniformly positive away from the edges of spectrum and away from zero by the formalism of Figotin–Pastur \cite{figotinpastur}. Bourgain-Schlag further explored this techniques, obtained a Large Deviation Theorem, and showed Anderson Localization for almost every phase at same regime of couplings and energies.

If one doesn't ask for uniform positivity, then Damanik-Killip \cite{damanikkillip} obtained that $L(E)$ is positive for almost every energies for all bounded measurable potentials via Kotani Theory. This is already enough to conclude that $H_x$ has empty absolutely continuous spectrum for almost every $x$.

Finally, we would like to mention that in \cite{aviladamanik}, by completely different techniques, we proved some general results regarding positivity of the Lyapunov exponent over general hyperbolic base dynamics which in particular imply the following results for doubling map generated potentials. For some fixed potential and fixed coupling, the Lyapunov exponent is positive away from a finte set of energies. Moreover, uniform positivity holds for some typical potentials. This is some work in progress.

We wish to point out that the question of uniform positivity for the doubling map was also noticed by Bourgain \cite{bourgain}, Damanik \cite{damanik}, Kr\"uger \cite{kruger}, Sadel-Schulzand \cite{sadelschulz}, and Schlag \cite{schlag}. In particular, Damanik \cite{damanik} proposed the following problem (which is \cite[Problem 5]{damanik}):

\begin{problem}
Find a class of functions $v\in L^\infty(\R/\Z)$ (other than trigonometric polynomials) such that for every $\l\ge \l_0(v)$, the Lyapunov exponent $L(E)$ obeys
$$
\inf_{E\in\R} L(E) \ge c\log \l
$$
for some suitable positive constant $c$.
\end{problem}

Hence, our Theorem~\ref{t.main} is a direct positive answer to this Problem. In fact, the lower bound in Theorem~\ref{t.main} is almost optimal.

\subsection{Strategy of the proof and further comments.} The proof follows Young's technique \cite{young} where the author constructed an open set $\CU\subset C^1(\R/\Z, \mathrm{SL}(2,\R))$ such that $(T,A)$ is \textit{nonuniformly hyperbolic} (i.e. positive Lyapunov exponent without being uniformly hyperbolic) for each $A\in\CU$. More concretely, the cocycle is of the form
\beq\label{eq:youngcocycle}
A_\e(x)=\begin{pmatrix}\l & 0\\ 0 & \l^{-1}\end{pmatrix}\cdot R_{\phi_\e(x)},
\eeq
where $\l>0$ is some large constant, $R_\gamma=\left(\begin{smallmatrix}\cos\gamma & -\sin \gamma\\ \sin\gamma & \cos\gamma \end{smallmatrix}\right)\in \mathrm{SO}(2,\R)$, and $\phi_\e\in C^1(\R/\Z, \R/(2\pi\Z))$ is some monotone function of degree one and is supported on some interval of length $\e$ for some small $\e>0$. Hence, the derivative of $\phi_\e$ is of order $\e^{-1}$ for many points on its support.

 In fact, it's also noted by Damanik \cite[Problem 6]{damanik} that one may apply Young's technique to Sch\"odinger cocycle and obtained some set of energies where $L(E)$ is positive. However, as he commented, due to the particular form of the cocycle maps \eqref{eq:youngcocycle}, it's not clear at all how one may apply her technique to Schr\"odinger cocycles and what kind of results one may obtain.

 The main novelties of this paper are thus the following two key observations. First, we observed that the very restricted form of cocycles \eqref{eq:youngcocycle} is not necessary. What's essential is that the derivative of the function attached to the rotation matrix is bounded away from zero. One may even allow a finite set of points of discontinuities.

 Secondly, by the approach of \cite{zhang} (which is also used in \cite{wangzhang}), we know that we may reduce the Schr\"odinger cocycle to its form of polar decomposition, which is very similar to \eqref{eq:youngcocycle}. However, it's well-known that Schr\"odinger cocycle is homotopic to identity while \eqref{eq:youngcocycle} is clearly not. That's why we need to assume the monotonicity of the potentials which necessary implies the existence of some points of discontinuities. Another difficulty is that when we consider the polar decomposition of the Schr\"odinger cocycle, $\l(x)$ in the diagonal matrices $\left(\begin{smallmatrix}\l(x) & 0\\ 0 & \l^{-1}(x)\end{smallmatrix}\right)$ while can be made uniformly large by some simple trick is, however, not a constant, see formula \eqref{eq:polar:form}. This in principle could be dangerous since large $\l(x)$ also implies large contraction which, when differentiating, may cancel the expanding effect of the base dynamics. Fortunately, the polar decomposition form of Schr\"odinger cocycles are so well balanced that the contracting effect can be very well compensated, see the proof of Lemma~\ref{l:deri:theta:n} for details.

Indeed, the competition between the hypebolicities of the base dynamics and of the cocycle map is the main difficulty to get positivity of the Lyapunov exponent. That's part of the reasons why most of the previous results were for small couplings. On the other hand, if one considers the base dynamics to be $x\to kx$ for some large $k$, then Bourgain and Bourgain-Chang \cite{bourgainbougain} obtained uniform positivity for suitable $C^1$ potentials and for any fixed coupling. Basically, the hyperbolicity of the base dynamics which is represented by $k$ beats the hypebolicity of the cocycle map which is represented by $\l$.

Finally, though we are currently only able to deal with the monotonic potentials at large couplings, we wish it builds a new bridge between techniques from dynamical systems and the spectral analysis of ergodic Sch\"odinger operators. In particular, we would like to further explore techniques in this paper to investigate the follow problems:

\begin{enumerate}
\item Show that uniform positivity of the Lyapunov exponent for all nonzero couplings. This makes the doubling map generated potentials behaves almost like potentials generated by i.i.d. random variables.
\item Remove the monotonicity and obtain uniform positivity for some $v\in C^r(\R/\Z,\R)$ which is not trigonometric polynomials. The existence of critical points clearly complicates the computation. However, it seems promising to further explore our techniques.
\item Prove some version of uniform Large Deviation Theorem (in energies) of the Lyapunov exponent. Then one may apply the approach of Bourgain-Schalg \cite{bourgainschlag} to get the H\"older continuity of the $L(E)$ (hence, the H\"older continuity of the integrated density of states via Thouless formula) and the Anderson Localization for almost every $x$.
\end{enumerate}
We are currently working on (3) in which the main difficulties we encountered are nothing other than the strong contracting effect of the cocycle, which is caused by the smallness of $\l(x)^{-1}$.

\section{Uniform positivity of the Lyapunov xxponent: proof of the Main Theorem}\label{s:proof:main:theorem}

Without loss of generality, we may assume that $v(0)=0$ and $\lim_{t\to 1-}v(t)=1$. Notice here in particular, we assume $v'(x)>0$ for all $x\in (0,1)$. Let $t=\frac E\l$ and $\CI=[-1,2]$. Define
\beq\label{eq:g}
g(x,t)\eqdef r^2+1=[t-v(x)]^2+1.
\eeq
Evidently, $c<g(x,t)<C$ for all $(x,t)\in\mathbb R/\mathbb Z\times\CI$. Define
\beq\label{eq:polar:form}
A(x;t,\l)\eqdef \Lambda(2x)\cdot O(x)=\begin{pmatrix}\lambda\sqrt{g(2x,t)} &0\\ 0 & \lambda^{-1}\sqrt{\frac1{g(2x,t)}}\end{pmatrix}\cdot R_{\theta(x;t)}
\eeq
where $R_\gamma=\begin{pmatrix}\cos\gamma & -\sin \gamma\\ \sin\gamma & \cos\gamma \end{pmatrix}\in \mathrm{SO}(2,\R)$ and $\t(x;t)$ is such that
\beq\label{eq:theta0}
\cot\t(x;t)=t-v(x).
\eeq

Let $L(t;\l)$ be the Lyapunov exponent of the cocycle $(T,A(\cdot;t,\l))$. Then by the same argument of \cite[Section 4.2]{zhang} or \cite[Section A.1]{wangzhang}, the following Theorem~\ref{t.B} implies Theorem~\ref{t.main}.

\begin{theorem}\label{t.B}
Let $v$ be as in the Theorem~\ref{t.main} and $A(x;t,\l)$ be as in \eqref{eq:polar:form}. Then there exists a $C_0=C_0(v)$ such that for all $\l>0$,
$$
L(t;\l)>\log\l-C_0 \mbox{ for all } t\in\CI.
$$
\end{theorem}
For the convenience of readers, we include in Appendix Section~\ref{s:a} the process that we convert the Schr\"odinger cocycle \eqref{eq:schrodinger:cocycle} to \eqref{eq:polar:form}.

Since all the following estimates will be uniform in $t\in\CI$, we will leave the dependence on $t$ implicit from now on. Let $\t_0(x)=\t(x)$ as defined in \eqref{eq:theta0}. Then inductively, we define the following functions $\phi_n, \theta_n:\R/\Z\to \R$ for $n\ge 1$:
\beq\label{eq:theta:phi}
\phi_n(x)\eqdef\cot^{-1}[\l^2g(T^nx)\cot\t_{n-1}(x)],\quad \t_n(x)\eqdef\phi_n(x)+\t(T^nx).
\eeq
Then the key to the proof of Theorem~B is to estimate the derivatives of $\phi_n(x)$ and $\theta_n(x)$ for all $n\ge 1$, which is done by the following lemma.

\begin{lemma}\label{l:deri:theta:n}
Let $\theta_n$, $n\ge1$ be as in \eqref{eq:theta:phi} and $\theta_0(x)=\theta(x)$. Then for all $t\in\CI$ and for each $n\ge 0$, it holds that
\beq\label{eq:deri:theta}
\frac{d\theta_n}{dx}(x)>c2^n
\eeq
for all $x\in\R/\Z$ where $\theta_n$ is differentiable.
\end{lemma}

\begin{proof}
For $n=0$, by \eqref{eq:theta0}, it clear that for all $t\in\CI$, it holds for any $x\neq0$ that
\beq\label{eq:theta0:deri}
\left|\frac{d\theta_0}{dx}(x)\right|=\frac{v'(x)}{1+(t-v(x))^2}>c.
\eeq
Suppose \eqref{eq:deri:theta} holds for $n=k$ and we consider $n=k+1$. Notice that $\theta_{k+1}(x)=\phi_{k+1}(x)+\t(T^{k+1}x)$. So we first consider $\phi_{k+1}(x)$. Clearly,
\begin{align}
\nonumber \frac{d\phi_{k+1}}{dx}(x)
=&
\frac{d}{dx}\cot^{-1}[\l^2g(T^{k+1}x)\cot\t_{k}(x)]\\
\label{eq:devi:phi1}=&
-\frac{\l^2\cot\t_{k}(x)\frac{dg(T^{k+1}x)}{dx}}{1+[\l^2g(T^{k+1}x)\cot\t_{k}(x)]^2}\\
\label{eq:devi:phi2}&-\frac{\l^2g(T^{k+1}x)
\frac{d\cot\t_{k}}{dx}}{1+[\l^2g(T^{k+1}x)\cot\t_{k}(x)]^2}.
\end{align}
Evidently, we have
\beq\label{eq:devi:phi2:lowerbound}
\eqref{eq:devi:phi2}=\frac{\l^2g(T^{k+1}x)(1+\cot^2\t_{k}(x))}{1+[\l^2g(T^{k+1}x)\cot\t_{k}(x)]^2}\cdot\frac{d\theta_k}{dx}(x)>0
\eeq
for all differentiable $x$. Moreover, for all $x$ such that $|g(T^{k+1}x)\cot\t_{k}(x)|<C\l^{-\frac32}$, it holds that
\beq\label{eq:devi:phi2:lowerbound2}
\eqref{eq:devi:phi2}>c\l\cdot\frac{d\theta_k}{dx}(x)>c2^{k}\l.
\eeq
Next, we consider $\eqref{eq:devi:phi1}$. In fact, it's more convenient to combine it with $\frac{d\t(T^{K+1}x)}{dx}(x)$. Notice that
$$
\csc\theta(x)=g(x)
$$
for all $x\in\R/\Z$. Hence,
$$
\frac{d\t(T^{k+1}x)}{dx}(x)+\eqref{eq:devi:phi1}
$$
becomes
\beq\label{eq:deri:phi1+theta}
\left[\frac{1}{g(T^{k+1}x)}-\frac{\l^2\cot\t_{k}(x)2(t-v(T^{k+1}x))}{1+[\l^2g(T^{k+1}x)\cot\t_{k}(x)]^2}
\right]2^{k+1}v'(T^{k+1}x)
\eeq
Let $h(x)=\l^2g(T^{k+1}x)\cot\t_{k}(x)$ and let $f(x)$ denotes the function in the square brackets of \eqref{eq:deri:phi1+theta}. Clearly, it holds for all $x\in\R/\Z$ that
\begin{align}\label{eq:f:upbound}
|f(x)|&<\left|\frac1{g(T^{k+1}x)}\right|+\left|\frac{2\l^2\cot\t_{k}(x)g(T^{k+1}x)[t-v(T^{k+1}x)]}{1+[\l^2g(T^{k+1}x)\cot\t_{k}(x)]^2}\right|\\
\nonumber &<1+\frac{4|h(x)|}{1+h^2(x)}\\
\nonumber &<3.
\end{align}
Moreover, for any $x$ such that $|g(T^{k+1}x)\cot\t_{k}(x)|>c\l^{-\frac32}$, which clearly implies $|h(x)|>c\l^{\frac12}$,
it holds that
\begin{align}\label{eq:f:lowerbound}
f(x)&>c-\frac{2|\l^2\cot\t_{k}(x)g(T^{k+1}x)[t-v(T^{k+1}x)]|}{1+[\l^2g(T^{k+1}x)\cot\t_{k}(x)]^2}\\
\nonumber &>c-\frac{C|h(x)|}{1+h^2(x)}\\
\nonumber &>c-C\frac{\l^{\frac12}}{1+\l}\\
\nonumber &>c.
\end{align}

Now if $|g(T^{k+1}x)\cot\t_{k}(x)|>c\l^{-\frac32}$, by \eqref{eq:f:lowerbound} and \eqref{eq:devi:phi2:lowerbound}, it's clearly that
\begin{align}\label{eq:theta:k+1:1}
\frac{d\t_{k+1}}{dx}
&=\left[\eqref{eq:devi:phi1}+\frac{d\t(T^{k+1}x)}{dx}(x)\right]+\eqref{eq:devi:phi2}\\
\nonumber &>f(x)2^{k+1}v'(T^{k+1}x)\\
\nonumber &>c2^{k+1}.
\end{align}
If $|g(T^{k+1}x)\cot\t_{k}(x)|<C\l^{-\frac32}$, by \eqref{eq:f:upbound} and \eqref{eq:devi:phi2:lowerbound2}, it holds that
\begin{align}\label{eq:theta:k+1:2}
\frac{d\t_{k+1}}{dx}
&=\eqref{eq:devi:phi2}+\left[\eqref{eq:devi:phi1}+\frac{d\t(T^{k+1}x)}{dx}(x)\right]\\
\nonumber &>c2^{k}\l-3\cdot 2^{k+1}v'(T^{k+1}x)\\
\nonumber &>c2^{k+1},
\end{align}
provided $\l$ is large. Clearly, \eqref{eq:theta:k+1:1} and \eqref{eq:theta:k+1:2} together imply \eqref{eq:deri:theta} for $n= k+1$, hence, concluding the proof of Lemma~\ref{l:deri:theta:n}.
\end{proof}

By Lemma~\ref{l:deri:theta:n}, $\theta_n$'s are piecewise $C^1$ monotone functions. Next we need to the following lemma.
\begin{lemma}\label{l:badpoints:upbound}
For each $t\in\CI$ and for each $n\ge 0$, the set of discontinuities of $\t_n$ is a subset of
\beq\label{eq:dicontinuity}
\DD_n\eqdef\left\{\frac{j}{2^n},\quad 0\le j\le 2^n-1\right\}.
\eeq

Moreover, let $\CC_n=\{x\in\R/\Z: \t_n(x)\in \pi\Z+\frac\pi2\}$ and $\mathrm{Card}(S)$ stands for cardinality of a set $S$. Then for all $t\in\CI$, it holds that
\beq\label{eq:badpoints}
\mathrm{Card} (\CC_n)\le 2^{n+1}-1,
\eeq

\end{lemma}

\begin{proof}
To show \eqref{eq:dicontinuity}, we first notice that $\DD_n$ is the set of discontinuities of $\theta(T^nx)$ and $g(T^nx)$, and that $\DD_n\subset\DD_{n+1}$ for all $n\ge 0$.

Clearly, the set of discontinuities of $\theta_0(x)=\cot^{-1}[t-v(x)]$ is $\DD_0=\{0\}$. Suppose the set of discontinuities of $\theta_{k}$ is a subset of that $\DD_k$. Then by definition, the set of discontinuities of
$$
\phi_{k+1}(x)\eqdef\cot^{-1}[\l^2g(T^{k+1}x)\cot\t_{k}(x)]
$$
is a subset of $\DD_{k+1}$. Hence the set of discontinuities of
$$
\theta_{k+1}(x)=\phi_{k+1}(x)+\theta(T^{k+1}x)
$$
must be a subset of $\DD_{k+1}$. By induction, \eqref{eq:dicontinuity} holds for all $n\ge 0$.

Next we consider \eqref{eq:badpoints}. We will again proceed by induction. Since $t-v(x)$ is uniformly bounded on $\CI\times\R/\Z$, $\theta_0(x)=\cot^{-1}[t-v(x)]=\frac\pi2$ if and only if $t-v(x)=0$ which happens at most once by monotonicity of $v$. It's clearly that
\beq\label{eq:theta:image}
|\t_0(\R/\Z)|<\pi,
\eeq
here and in the following, $|f(J)|$ denotes the Lebesgue measure of the image of a connected interval $J\subset\R/\Z$ under a continuous function $f:J\rightarrow\R$.

Now suppose $\mathrm{Card}(\CC_k)\le 2^{k+1}-1$. Notice that
\beq\label{eq:phi.k+1=theta.k}
\phi_{k+1}(x)\in \Z\pi+\frac\pi2 \Longleftrightarrow \t_k(x)\in\Z\pi+\frac\pi2.
\eeq
For each $0\le j\le 2^{k+1}-1$, let $I_{k+1,j}=[\frac{j}{2^{k+1}}, \frac{j+1}{2^{k+1}})$. By the proof of \eqref{eq:dicontinuity}, $\phi_{k+1}$ and $\t_{k+1}$ is $C^1$ on each $I_{k+1,j}$. It's clearly that
\begin{align}\label{eq:badpoints:I_j}
 |\t_{k+1}(I_{k+1,j})|
 &=\left|[\phi_{k+1}+\theta(T^{k+1}\cdot)](I_{k+1,j})\right|\\
 \nonumber &\le
 \left|\phi_{k+1}(I_{k+1,j})\right|+\left|\theta(T^{k+1}\cdot)(I_{k+1,j})\right|\\
 \nonumber &=\left|\phi_{k+1}(I_{k+1,j})\right|+|\theta_0(\R/\Z)|\\
 \nonumber &<
 \left|\phi_{k+1}(I_{k+1,j})\right|+\pi.
\end{align}
By Lemma~\ref{l:deri:theta:n}, $\t_{k+1}$ is monotone on $I_{k+1,j}$. Hence \eqref{eq:badpoints:I_j} and \eqref{eq:phi.k+1=theta.k} imply that
$$
\mathrm{Card}(\CC_{k+1}\cap I_{k+1,j})\le \mathrm{Card}(\CC_k\cap I_{k+1,j})+2
$$
for each $j$. However, notice that $c<\theta(T^{k+1}x)<\pi-c$ for all $x\in I_{k+1,j}$, which implies that from $\phi_{k+1}$ to $\theta_{k+1}$, the image of $I_{k+1,j}$ may only go up. Hence
\beq\label{eq:badpoints:I_j1}
\mathrm{Card}(\CC_{k+1}\cap I_{k+1,j})\le \mathrm{Card}(\CC_k\cap I_{k+1,j})+1,
\eeq
which in turn implies
\begin{align*}
\mathrm{Card}(\CC_{k+1})&=\sum^{2^{k+1}-1}_{j=0}\mathrm{Card}(\CC_{k+1}\cap I_{k+1,j})\\
&\le \sum^{2^{k+1}-1}_{j=0}\mathrm{Card}(\CC_{k}\cap I_{k+1,j})+2^{k+1}\\
&\le \mathrm{Card}(\CC_{k})+2^{k+1}\\
&\le 2^{k+1}-1+2^{k+1}\\
&=2^{k+2}-1,
\end{align*}
concluding the proof.
\end{proof}

The purpose of Lemma~\ref{l:deri:theta:n} and \ref{l:badpoints:upbound} is the following corollary:

\begin{corollary}\label{c:badset:upbound}
Let $\|\cdot\|_{\R\PP^1}$ denotes the distance to the nearest point(s) in $\pi\Z$. Let
$$
\CB_n(\delta)\eqdef \left\{x\in\R/\Z:\left\|\theta_n(x)-\frac\pi2\right\|_{\R\PP^1}<\delta\right\}.
$$
Then $\mathrm{Leb}(\CB_n(\delta))<C\delta.$
\end{corollary}
\begin{proof}
It suffices to consider $\delta>0$ small. Then by Lemma~\ref{l:badpoints:upbound}, $\CB_n$ consists of at most $2^{n+3}-3$ connected components. By Lemma~\ref{l:deri:theta:n}, on each such component, $\frac{d\theta_n}{dx}\ge c2^n$ which clearly implies that the measure of each component is at most $C\frac{\delta}{2^n}$. Consequently,
\[
\mathrm{Leb}(\CB_n(\delta))< 2^{n+3}\cdot \frac{2\delta}{c2^n}< C\delta.
\]
\end{proof}

Now we are ready to prove Theorem~\ref{t.B}. The following argument more or less follows those in \cite{young}.
\begin{proof}[Proof of Theorem~\ref{t.B}]
Let $\vec e_1=\binom{1}{0}$ and $\vec e_2=\binom{0}{1}$. Evidently, it holds that
\beq\label{eq:lowerbound:LE}
L(t;\l)\ge\lim_{n\to\infty}\frac1n\int_{\R/\Z}\log\left\|A_n(x)\vec e_1\right\|dx,
\eeq
where the existence of the limit above is due to the Oseledec's Multiplicative Ergodic Theorem and Lebesgue's dominated Theorem.

Let $\vec w_{n}(x)=A_n(x)\vec e_1$ and $\vec v_n(x)=R_{\theta(T^nx)}A_{n}(x)\vec e_1$ for $n\ge 0$. Clearly
$$
\vec w_{n+1}(x)=\Lambda(T^{n+1}x)\vec v_n(x),\ \vec v_n(x)=R_{\theta(T^nx)}\vec w_n(x).
$$
In particular, $\|\vec v_n(x)\|=\|\vec w_n(x)\|$ for all $x\in\R/\Z$.

By the definition of $\theta_n(x)$, it clearly holds that
$$
\vec v_n(x)=\cos\t_n(x)\|\vec v_n(x)\|\vec e_1+\sin\t_n(x)\|\vec v_n(x)\|\vec e_2.
$$
Then
\begin{align*}
\|\vec w_{n+1}(x)\|
&=\|\Lambda(T^{n+1}x)\vec v_n(x)\|\\
&\ge \left\|\l\cos\t_n(x)\cdot\|\vec v_n(x)\|\cdot \vec e_1 \right\|\\
&=\l\left|\cos\t_n(x)\right|\cdot \|\vec w_n(x)\|.
\end{align*}
Proceed by induction, we obtain for all $x\in\R/\Z$ and $n\ge 1$
$$
\log\|A_n(x)\vec e_1\|=\|w_n(x)\|\ge \l^n\cdot \prod^{n-1}_{k=0}\left|\cos\t_k(x)\right|\cdot
$$
which together with \eqref{eq:lowerbound:LE} clearly implies:
\beq\label{eq:lowerbound:LE2}
L(t;\l)\ge \log\l+\limsup_{n\to\infty}\frac1n\sum^{n-1}_{k=0}\int_{\R/\Z}\log|\cos\t_{k}(x)|dx.
\eeq
Now we need to bound $\int_{\R/\Z}\log|\cos\t_{k}(x)|dx$ from below where we will need Corollary~\ref{c:badset:upbound}. The following estimate will be independent of $k$. So we fix an arbitrary $k\ge 0$ and some $\delta>0$ small. Define
$$
J_{i}\eqdef\left\{x\in\R/\Z:\left\|\theta_k(x)-\frac\pi2\right\|_{\R\PP^1}< 2^{-i}\delta\right\},\ i\in\N.
$$

It's straightforward calculation to see that
\begin{align}\label{eq:int:thetak1}
\int_{J_0}\log|\cos\t_{k}(x)|dx&=\int_{J_0}\log\left|\sin\left(\t_{k}(x)-\frac\pi2\right)\right|dx\\
\nonumber &\ge\int_{J_0}\log\frac2\pi\left|\t_{k}(x)-\frac\pi2\right|dx\\
\nonumber &=\sum_{i\in\N}\int_{J_i\setminus J_{i+1}}\log\frac2\pi\left|\t_{k}(x)-\frac\pi2\right|dx\\
\nonumber &\ge\sum_{i\in\N}\mathrm{Leb}(J_i\setminus J_{i+1})\log\left(\frac2\pi\cdot 2^{-(i+1)}\delta\right)\\
\nonumber &\ge\sum_{i\in\N}-Ci2^{-i}\delta\\
\nonumber &\ge -C\delta.
\end{align}
Let
$$
J\eqdef(\R/\Z)\setminus J_0=\left\{x:\left\|\theta_k(x)-\frac\pi2\right\|_{\R\PP^1}\ge \delta\right\}.
$$
Then it's clearly that
\beq\label{eq:int:thetak2}
\int_{J}\log|\cos\t_{k}(x)|dx=\int_{J}\log\left|\sin\left(\t_{k}(x)-\frac\pi2\right)\right|dx>C\log\delta.
\eeq

By choosing $\delta=\frac13$, \eqref{eq:int:thetak1} and \eqref{eq:int:thetak2} imply that for all $k\in\N$:
\begin{align*}
\int_{\R/\Z}\log|\cos\t_{k}(x)|dx
&=\int_{J_0}\log|\cos\t_{k}(x)|dx+\int_{J}\log|\cos\t_{k}(x)|dx\\
&\ge -\frac13C-C\log 3\\
&\ge -C,
\end{align*}
which together with \eqref{eq:lowerbound:LE2} clearly implies for all $t\in\CI$:
$$
L(t;\l)>\log\l-C_0,
$$
concluding the proof of Theorem~\ref{t.B}, hence, the proof of Theorem~\ref{t.main}.
\end{proof}

\vskip .5cm

\appendix

\section{Polar decomposition of Schr\" odinger cocycle}\label{s:a}

For $B\in \mathrm{SL}(2,\mathbb R)$, it is a standard result that we can decompose it as $B=U_1\sqrt{B^tB}$, where $U_1\in \mathrm{SO}(2,\mathbb R)$ and $\sqrt{B^tB}$ is a positive symmetric matrix. We can further decompose $\sqrt{B^tB}$ as $\sqrt{B^tB}=U_2\Lambda U_{2}^t$, where $U_2\in \mathrm{SO}(2,\mathbb R)$ and $\Lambda=\left(\begin{smallmatrix}\|B\|&0\\ 0&\|B\|^{-1}\end{smallmatrix}\right)$, thus $B=U_1U_2\Lambda U_2^t$.

Consider a map $B\in C^r(\mathbb R/\mathbb Z, \mathrm{SL}(2,\mathbb R))$ for some $r\ge1$. Then, it can be decomposed as
$$
B(x)=U_1(x)U_2(x)\Lambda(x) U_2^t(x).
$$
By Lemma 10 of \cite{zhang}, $U_1(x)$, $U_2(x)$ and $\Lambda(x)$ are $C^r$ in $x$ as long
as $B(x)\notin\mathrm{SO}(2,\mathbb R)$ for all $x\in\R/\Z$. Let $U(x)=U_1(x)U_2(x)$ and $O(x)=U_2^t(Tx)(U_1U_2)(x)\in \mathrm{SO}(2,\mathbb R)$. Then, it's clear that
$$
U^{-1}(Tx)B(Tx)U(x)=\Lambda(Tx)O(x)
$$
which in particular implies:
\beq
L(T,B)=L(T,(\Lambda\circ T)\cdot O).
\eeq

\vskip .15cm

Back to the Schr\"odinger cocycles~\eqref{eq:cocycle:map}. Since we assumed
$$
v(0)=0,\ \lim_{t\to 1-}v(t)=1,
$$
we only need to consider $t\in \CI=[-1, 2]$ for large $\lambda$, see e.g. \cite[Lemma 11]{zhang}. Then we use a simple trick to ensure that $\|A^{(E-\l v)}(x)\|$ is uniformly of size $\l$ as $\l$ getting large: we instead consider $A^{(t,\l)}=PA^{(E-\l v))}P^{-1}$, where
$$
P=\begin{pmatrix}\sqrt{\lambda}^{-1}& 0\\0& \sqrt{\lambda}\end{pmatrix}
$$
This obviously does not change the dynamics of the cocycle. Thus

$$
A(x)=A^{(t,\lambda)}(x)=\begin{pmatrix}\lambda[t-v(x)]& -\lambda^{-1}\\\lambda & 0\end{pmatrix}.
$$
Now we decompose $A(x)$ to its polar decomposition form. We keep using $U_1, U_2, U, \Lambda, O$ as in the decomposition of $B(x)$ above. Let $r(x,t)=t-v(x)$. Then $r(x,t)$ is uniformly bounded on $\mathbb R/\mathbb Z\times\CI$. If we set
$$
a=a(x;t,\lambda)=r^2+1+\frac{1}{\lambda^4}+\sqrt{(r^2+1+\frac{1}{\lambda^4})^2-\frac{4}{\lambda^4}},
$$
then evidently $c<a(x;t,\l)<C$ for all $(x,t,\lambda)\in\mathbb R/\mathbb Z\times\CI\times[c, \infty)$. Then a direct computation shows that $\|A\|=\lambda\sqrt{\frac{a}{2}}$.

A direct computation shows
$$
U_2=\frac{1}{\sqrt{(a-\frac{2}{\lambda^4})^2+\frac{4}{\lambda^4}r(x)^2}}\begin{pmatrix} a-\frac{2}{\lambda^4}& \frac{2}{\lambda^2}r(x)\\ -\frac{2}{\lambda^2}r(x)& a-\frac{2}{\lambda^4} \end{pmatrix}
$$
For simplicity let
$$
f(x,t,\lambda)=\left(\sqrt{(a-\frac{2}{\lambda^4})^2+\frac{4}{\lambda^4}r(x)^2}\right)^{-1}.
$$
Then, it's straightforward computation that the corresponding upper-left entry of the corresponding $O(x)$ is
$$
c(x,t,\lambda,\alpha)=c_4\left[r(x)-\frac{2r(Tx)}{\lambda^2a(Tx)}+\frac{2r(x)}{\lambda^4a(Tx)}-\frac{4r(Tx)}{\lambda^6 a(x)a(Tx)}\right],
$$
where
$$
c_4=\sqrt{\frac{2}{a(x)}}f(Tx)f(x)a(Tx)a(x).
$$
Hence, we have $c(x,t,\infty,\alpha)=\frac{t-v(x)}{\sqrt{(t-v(x))^2+1}}$. Furthermore, it is not difficult to see that for any fixed $\alpha$,
\begin{center}
$c(x,t,\lambda,\alpha)\rightarrow c(x,t,\infty,\alpha)$ in $C^1((0,1)\times\CI,\mathbb R)$ as $\lambda\rightarrow\infty$.
\end{center}
Indeed, it is easy to see this reduces to the convergence of $a(x,t,\l)$ to $a(x,t,\infty)=2r^2+2$ in $C^1$ topology, which is immediate. Recall that
$$
g(x,t)=r^2+1=[t-v(x)]^2+1.
$$
Evidently, $c<g(x,t)<C$ for all $(x,t)\in\mathbb R/\mathbb Z\times\CI$.

In the proof of the Theorem~B, it is clear that the argument is stable in the $C^1$ perturbation. Thus, we could replace $c(x,t,\l,\a)$ by $c(x,t,\infty,\a)$ for large $\l$ since they are sufficiently close in $C^1$ norm. Thus we can consider the following cocycle map
$$
A(x;t,\l)\eqdef \Lambda(2x)\cdot O(x)=\begin{pmatrix}\lambda\sqrt{g(2x,t)} &0\\ 0 & \lambda^{-1}\sqrt{\frac1{g(2x,t)}}\end{pmatrix}\cdot R_{\theta(x;t)}
$$
as defined in \eqref{eq:polar:form}.

\section{Uniform positivity of the Lyapunov xxponent for trigonometric polynomials}\label{s:appendix.b}

It is known that Hermann's trick works for doubling map, see e.g. \cite[Footnote 8]{damanik} or \cite[Section 1]{bourgainbougain}. We include a proof here for the convenience of the readers.

For simplicity, we consider $v(x)=2\l\cos(2\pi x)$. Let $z_j=e^{2\pi i2^jx}$. Then it clear that
\begin{align*}
\left(\prod^{n-1}_{j=0}z_j\right) A_n^{(E-\l v)}(x)
&=
\prod^{0}_{j=n-1}\left[z_j\cdot\begin{pmatrix}E-\l(z_{j-1}-z_{j-1}^{-1}) & -1\\ 1 & 0\end{pmatrix}\right]\\
&=
\begin{pmatrix}E-\l(z^2_{n-1}-1) & -1\\ 1 & 0\end{pmatrix}\cdots \begin{pmatrix}E-\l(z^2_0-1) & -1\\ 1 & 0\end{pmatrix}\\
&=
\begin{pmatrix}\l^n+M^{(n)}_1 & M^{(n)}_2\\ M^{(n)}_3 & M^{(n)}_4\end{pmatrix},
\end{align*}
where $M^{(n)}_j=M^{(n)}_j(z)$, $j=1,2,3,4$, are polynomials in $z$ without the zero order (i.e. constant) terms. In particular, $M^{(n)}_j(0)=0$.

It's clear that $L_n(z):=\log\|A^{(E-\l v)}_n(z)\|=\log\|(\prod^{n-1}_{j=0}z_j)A^{(E-\l v)}_n(z)\|$ for $x\in\R/\Z$ since then $|z_j|=1$. Moreover, it's also standard fact that $L_n(z)$ is subharmonic on $\C$ for any $n\ge 0$. Thus, we get for all $E\in \R$ that
\begin{align*}
L(E)
&=
\lim_{n\to\infty}\frac1n\int_{\R/\Z} L_n(z)dx\\
&=
\lim_{n\to\infty}\frac1n\int_{\R/\Z} \log\left\|\begin{pmatrix}\l^n+M^{(n)}_1 & M^{(n)}_2\\ M^{(n)}_3 & M^{(n)}_4\end{pmatrix}\right\|dx\\
&=
\lim_{n\to\infty}\frac1n L_n(0)\\
&=
\lim_{n\to\infty}\frac1n\int_{\R/\Z} \log\left\|\begin{pmatrix}(-\l)^n & 0\\ 0 & 0\end{pmatrix}\right\|dx\\
&\ge
\log\l
\end{align*}
where the inequality follows from submean inequality of subharmonic functions and the fact that $M^{(n)}_j(0)=0$ for $j=,1,2,3,4$ and all $n\ge 0$.

It's clear from the proof above that it works for any trigonometric polynomials, though the lower bound might not be not as good. If one moves away from trigonometric polynomials, then this submean inequality is no longer working.

In the quasiperiodic potentials case, a more sophisticated argument by Sorets-Spencer \cite{sorets} works for any nonconstant real analytic potentials, see also \cite{zhang} for a more recent proof via Avila's global theory. However, both arguments rely on the fact that the rotation on the unit circle preserves the height when complexifying phases, i.e. $x+iy+\a$ and $x+iy$ have the same imaginary part. This is clearly not the case for the doubling map which sends $x+iy$ to $2(x+iy)$. In particular, if the radius of analyticity of $v$ is finite, then $T^j(x+iy)=2^j(x+iy)$ soon escapes the region of analyticity for $y\neq0$. So those arguments of \cite{sorets,zhang} break down completely.

\vskip0.4cm
%\noindent{\bf\textit{Acknowledgments.}} \addcontentsline{toc}{subsection}{Acknowledgments}  We would like to thank ...

\end{document}